\documentclass{amsart}
\usepackage{amsmath, amssymb,amsthm,amscd}
\usepackage[pdftex]{graphicx}
\usepackage{float}
\usepackage{pb-diagram}
\usepackage{amscd}
\usepackage{array}
\usepackage{hyperref}
\newtheorem{thm}{Theorem}[section]
\newtheorem{prop}{Proposition}[section]
\newtheorem{lem}{Lemma}[section]
\newtheorem{cor}{Corollary}[section]
\newtheorem{rem}{Remark}
\newtheorem{ex}{Example}
\begin{document}

\title{On the geometry of singular K3 surfaces with discriminant 3, 4 and 7}
\author{Taiki Takatsu}
\date{}
\begin{abstract}
We give construction of singular K3 surfaces with discriminant 3 and 4 as double coverings over the projective plane.
Focusing on the similarities in their branching loci, we can generalize this construction, and obtain a three dimensional
moduli space of certain K3 surfaces which admit infinite automorphism groups.
Moreover, we show that these K3 surfaces are characterized in terms of the configuration of the singular fibres and a global section,
and also in terms of periods.
\end{abstract}
\maketitle
\section{Introduction}
\indent
A compact complex surface $X$ is called a {\it{K$3$ surface}} if the canonical bundle is trivial and the irregularity is zero.
By Torelli theorem for K3 surfaces,
lattice theory can be applied to the classification of K3 surfaces, and also to the classification of their automorphism groups.
\\
\indent
For example, Shioda and Inose \cite{Shioda} showed that any singular K3 surface admits an infinite automorphism group.
Here, a K3 surface is said to be {\it{singular}} if its Picard number is $20$.
Moreover, they proved the following theorem:
\\[1ex]
{\textbf{Theorem}}
{\it{
There exists a bijective map from the set of equivalence classes of singular K$3$ surfaces
to the set of equivalence classes of positive definite oriented even lattices of rank $2$.}}
\\[1ex]
\indent
This map is given by $X \mapsto T_{X}$, where $T_{X}$ is the transcendental lattice of $X$.
So the transcendental lattices are important invariants for singular K3 surfaces.
In this way,
singular K3 surfaces have been studied in detail,
although there are, so far, only eleven known cases where the automorphism groups are actually calculated.
$T_{X}$ is denoted by
$$\left(
\begin{array}{cc}
a & b \\
b & c
\end{array}
\right).$$
The transcendental lattices of these K3 surfaces are given by the following table:
\begin{center}
  \begin{tabular}{r|r|r}
  No. & [a,b,c] &                              \\ [2pt]  \hline
   1  & [2,1,2] & Vinberg \cite{Vinberg} \\ [2pt] 
   2  & [2,0,2] & Vinberg \cite{Vinberg} \\ [2pt] \hline
   3  & [2,1,4] & Ujikawa \cite{Ujikawa} \\ [2pt] \hline
   4  & [2,0,4] & Shimada \cite{Shimada} \\ [2pt] 
   5  & [2,1,6] & Shimada \cite{Shimada} \\ [2pt] 
   6  & [2,0,6] & Shimada \cite{Shimada} \\ [2pt]  \hline
   7  & [4,2,4] & Keum and Kondo \cite{Kondo} \\ [2pt]  \hline
   8  & [2,1,8] & Shimada \cite{Shimada} \\ [2pt] 
   9  & [4,1,4] & Shimada \cite{Shimada} \\ [2pt] 
   10 & [2,0,8] & Shimada \cite{Shimada} \\ [2pt] \hline
   11 & [4,0,4] & Keum and Kondo \cite{Kondo}
  \end{tabular}
\end{center}
Vinberg \cite{Vinberg} calculated the automorphism groups of singular K3 surfaces for the cases of No.1 and 2.
Ujikawa \cite{Ujikawa} calculated the automorphism group of a singular K3 surface for the case of No.3.
Shimada \cite{Shimada} calculated the automorphism groups of singular K3 surfaces for the cases of No.4, 5, 6, 8, 9 and 10.
Keum and Kondo \cite{Kondo} calculated the automorphism groups of singular K3 surfaces for the cases of No.7 and 11.
For the cases of Keum and Kondo, these K3 surfaces are Kummer surfaces.
\\
\indent
In this paper, we focus on the singular K3 surfaces with discriminant $3$, $4$ and $7$.
This paper aims at finding common properties among these three K3 surfaces, which should give us clue to understand 
these surfaces more in depth. 
\\
\indent
In the case of discriminant 7,
Naruki \cite{Naruki} has already obtained some constructions of it.
One of the constructions is given by the double covering branched over a sextic curve.
In particular,
this double covering is obtained by the quotient of this K3 surface by the involution associated to an elliptic fibration.
Considering the action of the involution on the singular fibres,
he characterized the singularities of the sextic curve and obtained the equation explicitly.
In Section 5.2, we explain this construction.
In the case of discriminant 3 and 4,
we have similar constructions as double covering over the projective plane, which can be described explicitly,
and it turns out that these three sextic curves have certain common properties in their singularities.
For example, these three curves have only four singular points, and one of them is a $D_{4}$ singularity.
The lines passing through this point define an elliptic fibration on these K3 surfaces.
\\
\indent
Then we can generalize these constructions of the three singular K3 surfaces as the double coverings over
the projective plane.
Branching loci of the double coverings are sextic curves.
The sextic curves have only four singularities, and one of them is a $D_{4}$ singularity.
Hence the K3 surfaces obtained in this way have elliptic fibrations associated to the sextic curves,
and this elliptic fibrations induce the sublattices of their N\'{e}ron-Severi groups.
Therefore, by lattice theory,
we can also describe these K3 surfaces by means of a condition of their periods.
Then we obtain a three dimensional moduli space, denoted in Section 7 by
$\mathcal{B} \hspace{-,1em}\setminus \hspace{-,2em}\mathcal{H} .$
Moreover, applying the same argument as Naruki \cite{Naruki},
we show that the constructions of these K3 surfaces by the double coverings over the projective plane,  
the elliptic fibrations and the condition of their periods are equivalent to each other.
\\
\indent
The moduli space $\mathcal{B} \hspace{-,1em}\setminus \hspace{-,2em}\mathcal{H} $
is obtained from the constructions of K3 surfaces, which generalize the construction of 
the singular K3 surfaces with discriminant 3, 4 and 7,
and these three K3 surfaces admit infinite automorphism groups.
Therefore we expect that any K3 surface included in $ \mathcal{B} \hspace{-,1em}\setminus \hspace{-,2em}\mathcal{H}$
admits an infinite automorphism group.
To confirm this, we can apply the same argument as Shioda and Inose \cite{Shioda}.
On any singular K3 surface,
they constructed an elliptic fibration which has infinitely many sections.
These sections induce an infinite subgroup of their automorphism group.
In the case of K3 surface $X \in \mathcal{B} \hspace{-,1em}\setminus \hspace{-,2em}\mathcal{H},$
we can also construct such an elliptic fibration in the same way.
\\[2ex]
{\bf{Acknowledgments.}}
The author would like to thank my advisor Fumiharu Kato for many helpful conversations and advices.
He would also like to thank Professor Shigeyuki Kondo for very useful discussions.
\section{Lattices}
 A {\it{lattice}} ($L$,$\langle \:,\: \rangle$) is a free $\mathbb{Z}$-module of finite rank with a non-degenerate symmetric bilinear form $\langle \;,\; \rangle :L\times L \to \mathbb{Z}$.
The {\it{orthogonal group}} of $L$ is the one consisting of isometries of $L$, denoted by $O(L)$.
The {\it{dual lattice}} $Hom_{\mathbb{Z}}(L,\mathbb{Z})$ is denoted by $L^{*}$.
Since $\langle \;,\; \rangle$ is non-degenerate,
a map 
$$ L \to L^{*} :\: x \mapsto \langle x, \cdot \rangle $$
is injective.
Then we regard $L$ as a sublattice of $L^{*}$.
\\
For example, the {\it{hyperbolic lattice}} $U$ is the lattice of rank 2 with the Gram matrix
\\[1.5ex]
\indent
$\left(
\begin{array}{cc}
0 & 1 \\
1 & 0
\end{array}
\right) $.
\\[2ex]
Let $(L, \langle \;,\; \rangle)$ be a lattice.
A {\it{sublattice}} of $L$ is a subgroup $S$ of $L$ such that $(S, \langle \;,\; \rangle |_{S})$ is a lattice.
A sublattice $S$ of $L$ is said to be {\it{primitive}} if the quotient group $L/S$ has no torsion element.
The orthogonal complement of $S$ is denoted by $S^{\bot}$.
We denote by $S(\delta_{1},\delta_{2},... ,\delta_{k})$ the subgroup of $L$ generated by $S$ and $\delta_{i} \in L$ 
($i=1,2,...,k$).
\\
It is easy to show the following lemmata.
\begin{lem}\label{1}
Let $S$ be a sublattice of $L$,
Then the following conditions are equivalent.
\\
\indent
(i) $S$ is a primitive sublattice of $L$.
\\
\indent
(ii) There exist $\delta_{i} \in L$ $(i=1, 2,...\:k)$  such that $S(\delta_{1},\delta_{2}, ... ,\delta_{k})=L$.
\end{lem}
Let $\{e_{i}\}$ ($i=1,2,...r$) be a basis of a lattice $L$.
We denote by $d(L)$ the absolute value of determinant of a matrix $(\langle e_{i},e_{j} \rangle)_{i,j}$.
This is well-defined.
\begin{lem}\label{2}
Let $L$ and $S$ be lattices.
Assume that $L$ has a sublattice of finite index isomorphic to $S$.
Then $d(L)$ is given by the formula:
$$ d(S)=d(L)\times |L/S|^{2}.$$
\end{lem}
\indent A lattice $L$ is called an {\it{even lattice}} if $\langle x,x\rangle $ is even for any element $x$ of $L$.
For an even lattice $L$, a quotient group $L^{*}/L$ is denoted by $A_{L}$. The {\it{discriminant quadratic form}} of $L$ is defined to be
$$q_{L}: A_{L} \to \mathbb{Q}/2\mathbb{Z}\;,\; q_{L}(x+L):=\langle x,x\rangle \: mod \:2\mathbb{Z}.$$
The group $O(q_{L})$ is the one consisting of automorphisms of $A_{L}$ that preserves $q_{L}$.
\\[1ex]
In lattice theory, the following lemma is well-known (\cite{Nikulin}).
\begin{lem}\label{3}
Let $T$ be an indefinite even lattice with the signature $(t_{+},t_{-})$, and let $q_{T}$ be the discriminant quadratic form of $T$.
Assume that $rank(T)\geq l(A_{T})+2$, where $l(A_{T})$ is the length of $A_{T}$.
Then any even lattices $T^{'}$ with the signature $(t_{+}, t_{-})$ and $q_{T^{'}}=q_{T}$ is isomorphic to $T$.
Moreover, a canonical map $O(T)\to O(q_{T})$ is surjective.
\end{lem}
\indent A element $\delta$ of $L$ is called  a {\it{$(-2)$-root}} if $\delta^{2}=-2$.
A lattice is called a {\it{root lattice}} if it is a negative-definite even lattice and generated by $(-2)$-roots.
Any root lattice is isomorphic to a sum of irreducible root lattices of types $A_{l}$, $D_{m}$ and $E_{n}$
$(l\geq 1,\: m\geq 4,\: n=6,7,8)$.
\\
For a $(-2)$-root $\delta$, the map $s_{\delta}:L \to L$ is defined to be
$$s_{\delta}(x)=x+ \langle x,\delta \rangle \delta \,\,\,\, \mbox{for} \; x\in L.$$
By an easy calculation, it preserves the bilinear form of $L$, and satisfies $$s_{\delta}^{2}=1.$$
Hence this is an element of $O(L)$ and a reflection with respect to the hyperplane $r^{\bot} \in L \otimes \mathbb{R}$.
Let  $W(L)$ the subgroup of $O(L)$ generated by all reflections.
\\[1ex]
Let $L$ be a unimodular even lattice,
and $S$ be a primitive sublattice of $L$.
The orthogonal complement of $S$, denoted by $T$,
is also a primitive sublattice of $L$.
Let $H$ be a quotient group $L/S \oplus T$.
We denote by $p_{S}$ (resp. $p_{T}$) the projection map from $A_{S}\oplus A_{T}$ to $A_{S}$ (resp. $A_{T}$).
Then $p_{S}$ and $p_{T}$ satisfy the following Lemma.
\begin{lem}\label{4}
$p_{S}|_{H}:H \to A_{S}$ and $p_{T}|_{H}:H\to A_{T}$ are bijective maps.
\end{lem}
\begin{proof}
It is enough to show the injectivity of $p_{S}$ and $p_{T}$ because $|A_{S}|\cdot |A_{T}|=|H|^{2}$.
Let
$(x$ mod $S$, $y$ mod $T)\in H$
be an element of the kernel of $p_{s}|H$,
where $x\in S^{*}$, $y\in T^{*}$.
Since $x+y \in L$, 
$x \in S\subset L$,
$y \in L\cap T^{*}$ and $T$ is primitive, $y\in T$.
Hence, $(x$ mod $S$, $y$ mod $T)=0$.
In the same way, $p_{T}|H$ is injective.
\end{proof}
\noindent
Hence the map
$$ r_{S,T}=p_{T}\circ (p_{S}|H)^{-1}:A_{S} \to A_{T}$$
is an isomorphism of groups.
\section{K3 surfaces, elliptic fibrations and period map}
%
\subsection{K3 surfaces}
For a K3 surface $X$, the second cohomology group $H^{2}(X,\mathbb{Z})$ with the cup product is a unimodular even lattice with 
the signature $(3,19)$. Then $H^{2}(X,\mathbb{Z})$ is isomorphic to $U^{\oplus 3}\oplus E_{8}^{\oplus 2}$.
Let $\omega_{X}$ be a nowhere vanishing holomorphic $2$-form of $X$.
The {\it{N\'{e}ron-Severi lattice}} $NS(X)$ is defined by
$$NS(X):=\{ x\in H^{2}(X,\mathbb{Z}) \hspace{+,3em}| \hspace{+,3em}\langle x,\omega_{X} \rangle =0\}
=H^{2}(X,\mathbb{Z})\cap H^{1,1}(X,\mathbb{R}).$$
The N\'{e}ron-Severi lattice $NS(X)$ is a lattice by the cup product, which is also denoted by $S_{X}$.
The rank of $NS(X)$ is called the {\it{Picard number}} of $X$, denoted by $\rho(X)$.
The {\it{transcendental lattice}} $T_{X}$ of $X$ is defined by $NS(X)^{\bot}$.
\\[1ex]
By Riemann-Roch theorem and Serre duality, we get the following lemma.
\begin{lem}\label{5}
Let $X$ be a K$3$ surface, and let $\delta$ be an element of $S_{X}$ with $\delta^{2} \geq -2$.
Then $\delta$ or $-\delta$ is represented by an effective divisor.
\end{lem}
\begin{proof}
Let $L$ be a line bundle on $X$ which represents $\delta \in S_{X}$.
By Riemann-Roch theorem and Serre duality, we have
\begin{equation}
   \begin{split}
&dim H^{0}(X,\mathcal{O}(L)) + dim H^{2}(X,\mathcal{O}(L))\\
&=dim H^{0}(X,\mathcal{O}(L)) + dim H^{0}(X,\mathcal{O}(-L))\\
&\geq 2+\delta^{2}/2 \\
&\geq 1.
   \end{split}
\end{equation}
Hence we have $dimH^{0}(X, \mathcal{O}(L))>0$ or $dimH^{0}(X,\mathcal{O}(-L))>0$,
and $\delta$ or $-\delta$ is represented by an effective divisor.
\end{proof}
Let $\Delta (X)$ be the set
$$\Delta (X)=\{ \delta \in S_{X} \: | \: \delta^{2}=-2 \: \} ,$$
and let $\Delta (X)^{+}$ be the subset of $\Delta (X)$
$$\Delta (X)^{+}:=\{ \delta \in \Delta (X)\: | \: \delta \; \mbox{is represented by an effective divisor} \: \} .$$
We denote by $P^{+}(X)$ the irreducible component of the set 
$$P(X)=\{ x \in H^{1,1}(X,\: \mathbb{R})\: |\: x^{2} >0\} $$
which contains a K\"{a}hler class, and this is called the {\it{positive cone}}.
The group $W(X)$ acts on the space $P^{+}(X)$ in a canonical way.
It is known that one of the fundamental domains is given by
$$D(X)=\{ x \in P^{+}(X) \: |\: \langle x,\delta \rangle >0,\: \forall \delta \in \Delta (X)^{+} \}.$$
This is called the {\it{K\"{a}hler cone}}.
\\[1ex]
For a K3 surface $X$, the following theorem is well-known (\cite{BR, PS}).
\begin{thm}$($The Torelli theorem for K$3$ surfaces$)$
\\
Let $X$, $X^{'}$ be K$3$ surfaces, and let $\omega_{X}$ (resp.$\omega_{X^{'}}$) be a nowhere vanishing holomorphic $2$-form on $X$
(resp.$X^{'}$).
Assume that the isomorphism $\phi :H^{2}(X,\mathbb{Z}) \to H^{2}(X^{'},\mathbb{Z})$ satisfies the following conditions:
$$(a)\; \phi(\omega_{X}) \in \mathbb{C} \omega_{X^{'}},$$
$$(b)\: \phi(D(X))=D(X^{'}).$$
Then there exists uniquely an isomorphism $\psi :X{'} \to X$ such that $\psi^{*}=\phi$.
\end{thm}
%
\subsection{Elliptic fibration}
An {\it{elliptic K$3$ surface}} is the pair of a K3 surface $X$ and an elliptic fibration $f: X \to \mathbb{P}^{1}$.
Let $f:X \to \mathbb{P}^{1}$ be an elliptic K3 surface with a global section $g$.
We regard each smooth fibre as an elliptic curve with the zero element being the intersection with $g$.
In particular, we have a rational selfmap of $X$ given by inversion on each fiber.
By the minimality of a K3 surface, this rational map extends to  an automorphism of $X$.
This automorphism is called the {\it{inversion involution}}.
\\[2ex]
Let $F_{v}=\sum_{i=0}^{m_{v}} C_{v}^{i}$ $(1\leq v \leq k)$ be singular fibres of $f$,
where $C_{v}^{i}$ is an irreducible component of $F_{v}$.
We assume that a global section $g$ intersects the rational curve $C_{v}^{0}$.
By Kodaira's classification of singular fibres,
a lattice generated by $C_{v}^{i}$ $(1 \leq i \leq m_{v})$ is isomorphic to 
the irreducible root lattice $L_{v}$.
We denoted by $E$ a line bundle represented by a general fibre of $f$,
and by $F$ a line bundle $E+ [g]$, where $[g]$ is a line bundle represented by $g$.
Since $E^{2}=F^{2}=0$ and $E\cdot F=1$,
a lattice generated by $E$ and $F$ is isomorphic to the hyperbolic lattice $U$.
Since $C_{v}^{i}\cdot E=C_{v}^{i}\cdot F=C_{\lambda}^{i}\cdot C_{\mu}^{j}=0$ $(\lambda \neq \mu)$,
these lattices are orthogonal to each other.
Hence $S_{X}$ has a sublattice isomorphic to $U \oplus(\oplus_{v=1}^{k} L_{v})$.
We call this sublattice the {\it{trivial sublattice}}, denoted by $L_{f}$.
This means that $L_{f}$ is determined by the types of singular fibres of $f$.
\subsection{Period map}
Let $L$ be a unimodular even lattice with the signature $(3,19)$.
The {\it{period domain}} $\Omega$ is defined by
$$\Omega=\{\omega \in L \otimes \mathbb{C}\:|\: 
\langle \omega, \omega \rangle =0, \: \langle \omega, \overline{\omega} \rangle >0\} .$$
Let $X$ be a K3 surface, and let $\alpha_{X}: H^{2}(X,\mathbb{Z})\to L$ be an isometry of lattices.
The pair ($X,\alpha_{X}$) is called a {\it{marked K$3$ surface}}.
By Riemann condition, we have $\alpha_{X}(\omega_{X})\in \Omega$.
\\[1ex]
Let $\mathcal{M}$ be the set of equivelence classes of marked K3 surfaces.
The period map of marked K3 surfaces is defined to be:
$$\lambda:\mathcal{M} \to \Omega ,\;  (X,\alpha_{X}) \to \alpha(\omega_{X}).$$
It is known that this map is surjective (\cite{T}).
\section{The result of Shioda and Inose}
Let $\Phi:X \to \mathbb{P}^{1}$ be an elliptic K3 surface,
and let $T$ be the trivial sublattice.
The sections of $\Phi$ generate a group, called the {\it{Mordell-Weil group}} $MW(\Phi)$.
The Mordell-Weil group $MW(\Phi)$ is a finitely generated group, and satisfies the following isomorphism (\cite{Shioda2}):
\begin{equation}\label{eq:MW}
MW(\Phi) \cong NS(X)/T
\end{equation}
We denote by $r(\Phi)$ the rank of $MW(\Phi)$.
When $r(\Phi)=0$, let $n(\Phi)$ denote the order of $MW(\Phi)$.
\\[2ex]
We will use the following theorem and lemmata by Shioda and Inose \cite{Shioda}.
%
%
\begin{lem}\label{7}$($Shioda and Inose $[4]$, lemma $1.1$$)$
\\
Assume that an effective divisor $D$ on a K$3$ surface $X$ has the same type as a simple singular fibre of an elliptic surface in the sense of Kodaira $[3]$  \S $6$.
Then there is a unique elliptic pencil $\Phi:X\to \mathbb{P}^{1}$ of which $D$ is a singular fibre. Moreover, any irreducible curve $C$ on X with
 $(CD)=1$ defines a (holomorphic) section of $\Phi$.
\end{lem}
\begin{lem}\label{8}$($Shioda and Inose $[4]$, lemma $1.2$$)$
\\
If a divisor $D$ on an elliptic surface has its support contained in a (simple) singular fibre, then the self-intersection number $D^2$ is non-positive, and $D^2=0$ if and 
only if $D$ is a multiple of a singular fibre.
\end{lem}
\begin{lem}\label{9}
$($Shioda and Inose $[4]$, lemma $1,3$$)$
\\ Let $\Phi:X\to \mathbb{P}^{1}$ denote an elliptic K$3$ surface, and $D_{v}=\Phi^{-1}(t_{v}) \; (1\leq v \leq k)$ be the singular fibres of $\Phi$.
We denote by $m_{v}$ and $m_{v}^{(1)}$ respectively the number of irreducible components of $D_{v}$,
 and the number of components of $D_{v}$ of which multiplicity is $1$. Assume that $\Phi$ has a section.
Then
\\
$(i)$ The Picard number  $\rho(X)$ of X
 is given by the following formula:
 \\[1ex]
 $(4.1)$  \indent $\rho(X)=r(\Phi)+2+\sum_{v=1}^k (m_{v}-1)$.
\\[2ex]
$(ii)$ When $r(\Phi)=0$, we have 
\\[1ex]
$(4.2)$ \indent $ |\mathop{\mathrm{det}}T_{X}|=|\mathop{\mathrm{det}}S_{X}|=(\prod_{v=1}^k m_{v}^{(1)})/n(\Phi)^{2}$.
\end{lem}
As in introduction,
they showed the theorem.
\begin{thm}\label{10}(Shioda and Inose $[4]$)
\\
There exists a bijective map from the set $\mathcal{S}$ of equivalence classes of singular K$3$ surfaces
to the set $\mathcal{T}$ of equivalence classes of positive definite oriented even lattices of rank $2$.
\end{thm}
The map is defined to be
$X \mapsto T_{X}$.
By Torelli theorem for K3 surfaces,
it is easy to show the injectivity of this map.
For the surjectivity,
Shioda and Inose explicitly constructed singular K3 surfaces.
\\[1ex]
Moreover they gave another construction for singular K3 surfaces with
discriminant $3$ and $4$. We will explain this construction in the next section.
%
%
%
\section{The construction of  \texorpdfstring{$X_{3}$, $X_{4}$ and $X_{7}$}{X3, X4 and X7}}
Let $\mathcal{T}$ be a set of equivalence classes of positive definite oriented even lattice of rank $2$.
There is the classification table of $\mathcal{T}$ in terms of discriminants in \cite{Conway}.
According to this table,
there exists uniquely an element of $\mathcal{T}$ with discriminant 3, 4 and 7, respectively. This means the singular K3 surfaces with discriminant $d$ is unique up to isomorphisms for $d=3,\:4,\:7$,
so we denoted by $X_{d}$ these K3 surfaces, respectively.
\begin{rem}
In general, there exist singular K$3$ surfaces which have the same discriminant,
but is not isomorphic to each other.
\end{rem}
\subsection{The construction of \texorpdfstring{$X_{3}$ and $X_{4}$}{X3 and X4}}
This construction is given by Shioda and Inose (\cite{Shioda}).
\\[1ex]
(i) Let $E_{3}$ be an elliptic curve $\mathbb{C}/\mathbb{Z}+\mathbb{Z}\omega$,
where $\omega=\exp(2\pi \sqrt{-1}/3)$,
and let $A=E_{3}\times E_{3}$.
The automorphism of $A$ is defined to be:
$$ \sigma:A\to A,\; \sigma(x,y)=(\omega x,\omega^{2} y).$$
Then the minimal resolution $Y_{3}$ of the quotient $A/\sigma$ is a singular K3 surface isomorphic to $X_{3}$.
\\[1ex]
(ii) Let $E_{4}$ be an elliptic curve $\mathbb{C}/\mathbb{Z}+\mathbb{Z}i$, where $i=\sqrt{-1}$,
and let $A=E_{4}\times E_{4}$.
The automorphism of $A$ is defined to be:
$$ \sigma:A\to A,\; \sigma(x,y)=(i x,-iy).$$
Then the minimal resolution $Y_{4}$ of the quotient $A/\sigma$ is a singular K3 surface isomorphic to $X_{4}$.
\\
\indent
Shioda and Inose constructed an elliptic fibration on $Y_{k}$ $(k=3,4)$ for which the trivial lattice associated to is of rank 20.
Hence $Y_{k}$ is a singular K3 surface.
By the formula (4.2), they calculated the discriminant of $Y_{k}$, respectively.
%
%
\subsection{The construction of  \texorpdfstring{$X_{7}$}{X7} (the result of Naruki)}
Here we give the result of Naruki \cite{Naruki}.
\\
We denote $\exp(2\pi \sqrt{-1}/7)$ by $\zeta$.
Let $\mathfrak{p}$ be a principal ideal of the cyclotomic field $\mathbb{Q}(\zeta)$ generated by $1-\zeta$ which includes $7$.
\\[1ex]
Naruki defined a Hermitian form $H_{7}$:
$$  H_{7}(z)=z_{1} \overline{z_{1}}+z_{2} \overline{z_{2}} +(\zeta+ \overline{\zeta})z_{3} \overline{z_{3}}
\; \,\,\, \mbox{for} \: z=(z_{1},z_{2},z_{3})\in \mathbb{C}^{3}.$$
Let $B_{7}$ be the set
$$B_{7}:= \{ (z_{1}:z_{2}:z_{3}) \in \mathbb{P}^{2}\: |\; H_{7}(z)<0 \}. $$
Since the signature of $H_{7}$ is (2,1),
$B_{7}$ is isomorphic to the complex $2$-ball.
Let $SU_{7}$ be the group consisting of $(3,3)$-matrices with determinant 1 which is unitary with respect to $H_{7}$.
The group $SU_{7}$ acts $B_{7}$.
Let $\Gamma_{7}$ be the subgroup of $SU_{7}$ consisting of elements whose entries are integers of $\mathbb{Q}(\zeta)$,
and let
$\Gamma_{7}^{'}$ (resp. $\Gamma_{7}^{''}$) be a subgroup of $\Gamma_{7}$
consisting of elements which are congruent to identity modulo $\mathfrak{p}$ (resp. $\mathfrak{p}^{2}$).
Then Naruki's main result is as follows:
\begin{thm}(Naruki $[1]$, theorem $1$)
\\
The quotient surface $B_{7}/\Gamma_{7}^{'}$ is a singular K$3$ surface isomorphic to $X_{7}$.
The branching locus of $B_{7}/\Gamma_{7}^{''} \to B_{7}/\Gamma_{7}^{'}$ consists of twenty-eight rational curves,
and the group $\Gamma_{7}^{''}/\Gamma_{7}^{'}$ acts on these curves as permutations.
\end{thm}
Moreover, Naruki obtained another constructions of $B_{7}/\Gamma_{7}^{'}$ as follows.
Let $S(7)$ be the elliptic modular surface of level 7 (\cite{Shioda3}). This surface has an elliptic fibration $F:S(7) \to X(7)$,
where X(7) is the Klein quartic curve.
Naruki showed that $B_{7}/\Gamma_{7}^{'}$ is a quotient of $S(7)$,
and the fibration $F$ induces an elliptic fibration $f:B_{7}/\Gamma_{7}^{'} \to \mathbb{P}^{1}$,
which has three singular fibres of types $I_{7}$ and others are of type $I_{1}$.
$$
 \begin{CD}
   S(7) @>{F}>>X(7) \\
  @VVV @VVV \\
    B_{7}/\Gamma_{7}^{'} @>{f}>>\mathbb{P}^{1}
 \end{CD}
$$
Let $\iota$ be the involution associated to  $f:B_{7}/\Gamma_{7}^{'} \to \mathbb{P}^{1}$.
Blowing down $(-1)$-curves of the quotient surface of $B_{7}/\Gamma_{7}^{'}$ by $\iota$,
Naruki obtained a construction of $X_{7}$ by the double covering branched over projective plane.
The branching locus of this is given by the equation as follows:
$$(x_{0}^2x_{1}+x_{1}^2x_{2}+x_{2}^2x_{0}-3x_{0}x_{1}x_{2})^2-4x_{0}x_{1}x_{2}(x_{0}-x_{1})(x_{1}-x_{2})(x_{2}-x_{0})=0,$$
which has a $D_{4}$ singularity at (1,1,1) and $A_{4}$ singularities at (1,0,0), (0,1,0) and (0,0,1).
\section{The double plane model of a K3 surface}
%
\subsection{Sextic curves with a  \texorpdfstring{$D_{4}$}{D4} singularity and elliptic fibrations}
Let $C$ be a sextic curve on $\mathbb{P}^{2}$. We assume its singular points are of ADE-type.
Moreover we assume $C$ has a singularity, denoted by $p$, of type $D_{4}$.
We denote by $X_{C}$ the  minimal resolution of the double covering branched over $C$.
Then $X_{C}$ is a K3 surface. We call $C$ the branching locus of $X_{C}$.
We denote by  $\widetilde{\mathbb{P}^{2}}$ the surface given by blowing up $\mathbb{P}^{2}$ at $p$,
$E$ the exceptional divisor on $\widetilde{\mathbb{P}^{2}}$
and $\widetilde{C}$ the inverse image of $C$.
A canonical projection from $\widetilde{\mathbb{P}}^{2}$ to $E$, denoted by $\phi : \widetilde{\mathbb{P}}^{2} \to E$,
induces the projective bundle on $E$ $(\cong \mathbb{P}^{1})$ as given Figures \ref{blowingup1} and \ref{blowingup2}.
\begin{figure}[H]
   \centering
  \begin{minipage}{0.4\columnwidth}
   \centering
   \includegraphics[width=\columnwidth]{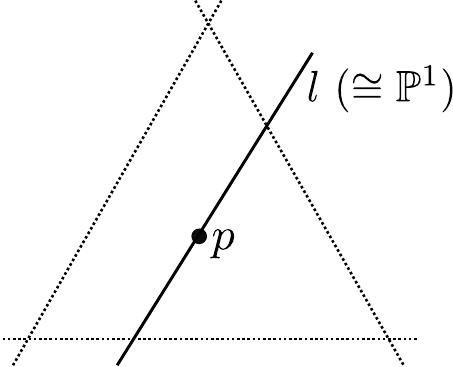}
   \caption{}
   \label{blowingup1}
  \end{minipage}
  \begin{minipage}{0.4\columnwidth}
   \centering
   \includegraphics[width=\columnwidth]{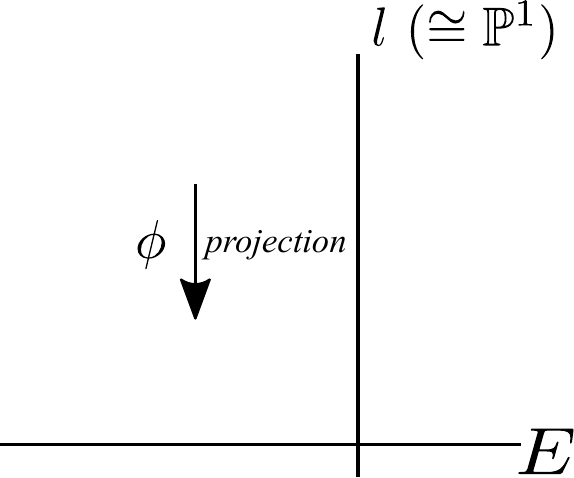}
   \caption{}
   \label{blowingup2}
  \end{minipage}
\end{figure}%
Since the general lines passing through $p$ intersects $C$ three times  outside the triple point $p$,
the inverse image of these lines are elliptic curves on $X_{C}$.
We denote by $g$ the rational curve on $X_{C}$ corresponding to $E$.
Since  these elliptic curves intersect $g$,
these family of elliptic curves defines an elliptic fibration of $X_{C}$ associated to $C$, denoted by $\Phi:X_{C}\to \mathbb{P}^{1}$
, and $g$ is a global section of $\Phi$.
Considering a natural correspondence from $X_{C}$ to $\widetilde{\mathbb{P}}^{2}$,
we have a diagram as follows:
$$
 \begin{CD}
   X_{C} @>{2:1}>>{\widetilde{\mathbb{P}}^{2}} \\
  @V{\Phi}VV @V{\phi}V{projection}V \\
    \mathbb{P}^{1} @>{\cong}>>E
 \end{CD}
$$
Let $l_{0}$ be a line on $\mathbb{P}^{2}$ passing through $p$ and
let $\widetilde{l_{0}}$ be the strict transform $l_{0}$ under the blowing-up $\pi :X_{C} \to \mathbb{P}^{2}$.
We assume that $\widetilde{l_{0}}$ intersects a singular point of $\widetilde{C}$.
Since the number of components of $\pi^{-1}(l_{0})$ is larger than $1$,
$\pi^{-1}(l_{0})$ is a singular fibre of $\Phi$.
\begin{figure}[H]
 \centering
 \includegraphics[width=8cm]{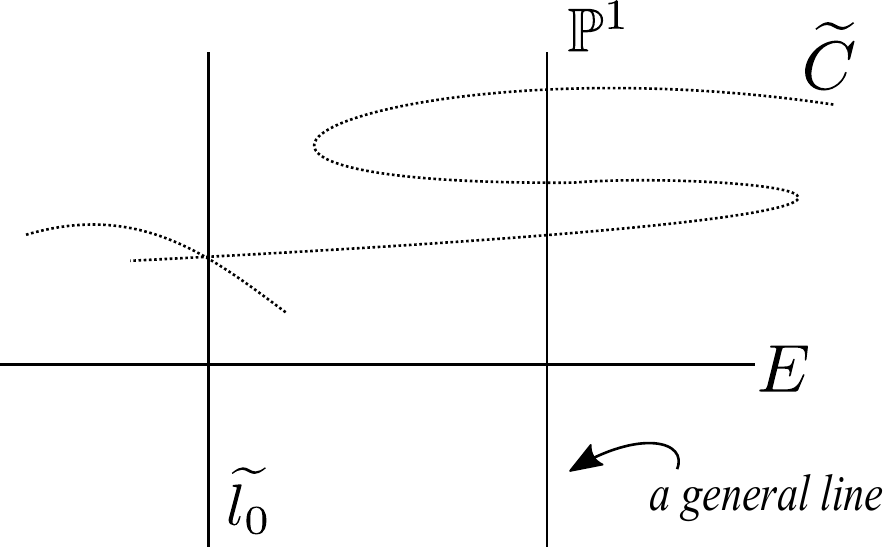}
 \caption{$C$ and lines passing through $p$}
 \label{fig:3}
\end{figure}
\subsection{Singular K3 surfaces with discriminant 3, 4, and 7}
We denote by $X_{d}$ the singular K3 surface of discriminant $d$ and by $C_{d}$ a branching locus if it can be constructed by the double covering $f:X \to \mathbb{P}^{2}$.
As in 5.2, for the case $d=7$, Naruki obtained the branching locus  in [1]. $C_{7}$ is given by the equation:
$$(x_{0}^2x_{1}+x_{1}^2x_{2}+x_{2}^2x_{0}-3x_{0}x_{1}x_{2})^2-4x_{0}x_{1}x_{2}(x_{0}-x_{1})(x_{1}-x_{2})(x_{2}-x_{0})=0.$$
Since $C_{7}$ has a $D_{4}$ singularity at $(1,1,1)$, there exist an elliptic fibration $f:X \to \mathbb{P}^{1}$ associated to $C$.
As in 5.2, this fibration equals the fibration $f:B_{7}/\Gamma_{7}^{'} \to \mathbb{P}^{1}$,
which has three singular fibres of types $I_{7}$ and others are of type $I_{1}$.
Then the trivial sublattice is isomorphic to $U\oplus A_{6}^{\oplus 3}$, denoted by $L_{f}$,
and the N\'{e}ron-Severi lattice $NS(X_{7})$ has a sublattice of finite index isomorphic to $L_{f}$.
In the case of $d=3$ and 4, we have similarly the following results.
\begin{prop}\label{12}
$X_{d} \; (d=3,4)$ can be constructed  by the double covering.
\\
$(i)$ $C_{3}$ is given by the equation 
$$(x_{0}-x_{1})(x_{1}-x_{2})(x_{2}-x_{0})\{(x_{0}+x_{1}+x_{2})^3+(x_{0}-x_{1})(x_{1}-x_{2})(x_{2}-x_{0})\}=0,$$
which has a $D_{4}$ singularity at $(1,1,1)$ and $A_{5}$ singularities at  the intersections of lines and elliptic curve of components of $C_{3}$. 
In particular, $NS(X_{3})$ has a sublattice of finite index isomorphic to
 $U\oplus E^{\oplus 3}_{6}$
\\
$(ii)$  $C_{4}$ is given by the equation 
$$ x_{0}x_{1}x_{2}(x_{0}-x_{1})(x_{1}-x_{2})(x_{2}-x_{0})=0,$$ which has $D_{4}$ singularities at intersections of six lines.
In particular, $NS(X_{4})$ has a sublattice of finite index isomorphic to $U\oplus D^{\oplus 3}_{6}$.
\end{prop}
\begin{proof}
By an easy calculation, it is easy to show these sextic curves have  singularity as above and are smooth elsewhere.
We denote by $X_{d}^{'} \; (d=3,4)$ the K3 surface which are double covering  branching over these  sextic curves.
Since $C_{d}$ have a $D_{4}$ singularity at $(1,1,1)$,
they have an elliptic fibration associated to $C$, denoted by $\Phi_{d}^{'}:X_{d}^{'} \to \mathbb{P}^{1}$.
We define $\widetilde{\mathbb{P}}^{2}$, $E$, $\widetilde{C_{d}}$ and $\pi:X_{d}^{'} \to \mathbb{P}^{2}$ in the same way as 6.1.
We have a diagram as follows:
$$
 \begin{CD}
   X_{d}^{'} @>{2:1}>>\widetilde{\mathbb{P}}^{2} \\
  @V{\Phi_{d}^{'}}VV @V{\phi}V{projection}V \\
    \mathbb{P}^{1} @>{\cong}>>E
 \end{CD}
$$
The lines $ l_{ij}$ given by $x_{i}-x_{j}=0 \; (i\neq j)$ are components of a branching locus.
we denote by $\widetilde{l_{ij}}$ the strict transform of $l_{ij}$ under the blowing-up.
These curves pass through singular points of $\widetilde{C_{d}}$.
Hence $\pi^{-1}(l_{ij})$ is a singular fibre of $\Phi_{d}^{'}$.
\\ 
In the case where $d$ is 3, 
there are two singular points $p_{ij}$ and $q_{ij}$ of $\widetilde{C_{3}}$ on $\widetilde{l_{ij}}$ as given in Figure \ref{X3fibre1}.
\begin{figure}[H]
   \centering
   \includegraphics{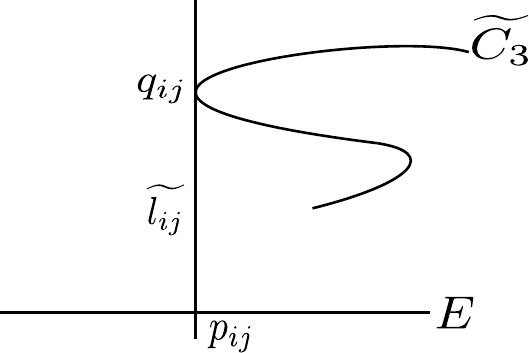}
   \caption{}
   \label{X3fibre1}
\end{figure}%
\noindent
Here, $p_{ij}$ (resp. $q_{ij}$) is an intersection of $\widetilde{l_{ij}}$ and $E$ (resp. the elliptic curve of component of $\widetilde{C_{3}}$).
Since $\widetilde{C_{3}}$ has an $A_{1}$ (resp. $A_{5}$) singularity at $p_{ij}$ (resp. $q_{ij}$),
the dual graph of components of $\pi^{-1}(l_{ij})$ is given by Figure \ref{X3fibre2}:
\begin{figure}[H]
  \centering
  \includegraphics[width=5cm]{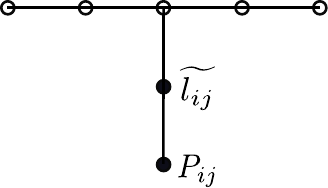}
  \caption{}
  \label{X3fibre2}
\end{figure}%
\noindent
Here, white circles are obtained by resolution at $q_{ij}$
and $P_{ij}$ is obtained by resolution at $p_{ij}$.
Hence we conclude types of $\pi^{-1}(l_{ij}) \; (i\neq j)$ are $I\hspace{-,1em}V^{*}$.
Since $X_{3}^{'} $ has a global section and three singular fibres of type $I\hspace{-,1em}V^{*}$,
the trivial sublattice is isomorphic to $U\oplus E_{6}^{\oplus 3}$.
In particular, $X_{3}^{'}$ is a singular K3 surface and $r(\Phi_{3}^{'})=0$.
Now, $X_{3}^{'} $ has three global sections. One of them is denoted by $g$ corresponding to $E$, and
others are obtained by the inverse image of the equation: $ x_{0}+x_{1}+x_{2}=0$.
From the formula (4,2), we deduce that
$$|\mathop{\mathrm{det}}T_{X_{3}^{'}}|=\frac{27}{n(\Phi_{3}^{'})^{2}}=3.$$
By \cite{Conway}, the binary quadratic form of discriminant 3 is unique.
By Torelli theorem, we conclude $X_{3}^{'} $ is isomorphic to $X_{3}$.
\\
In the case of $d$ is 4,
there are three singular points of $\widetilde{C_{4}}$ on $\widetilde{l_{ij}}$, denoted by $p_{ij},\: q_{ij}$ and $r_{ij}$, as given in Figure \ref{X4fibre1}.
\begin{figure}[H]
  \centering
  \includegraphics[width=5cm]{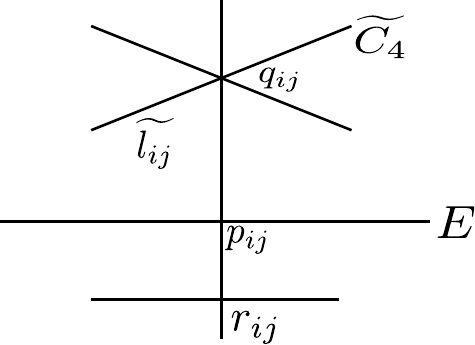}
  \caption{}
  \label{X4fibre1}
\end{figure}%
\noindent
$\widetilde{C_{4}}$ has $A_{1}$ singularities at $p_{ij}$ and $r_{ij}$,
and a $D_{4}$ singularity at $q_{ij}$.
Hence the dual graph of components of $\pi^{-1}(l_{ij})$ is given by Figure \ref{X4fibre2}:
\begin{figure}[H]
  \centering
  \includegraphics[width=5cm]{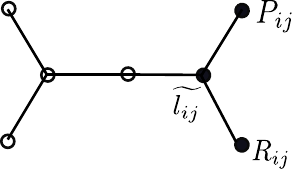}
  \caption{}
  \label{X4fibre2}
\end{figure}%
\noindent
Here, white circles are obtained by resolution at $q_{ij}$
and $P_{ij}$ (resp. $R_{ij}$) is obtained by resolution at $p_{ij}$ (resp. $q_{ij}$).
Hence we conclude types of $\pi^{-1}(l_{ij})$ $(i \neq j)$ are $I_{2}^{*}$.
Since the trivial sublattice is isomorphic to $U\oplus D_{6}^{\oplus 3}$,
$X_{4}^{'}$ is a singular K3 surface.
Now, $X_{4}^{'}$ has four global sections.
One of them is denoted by $g$ corresponding to $E$, and
others are obtained by the inverse image of the equation: $x_{i}=0$ ($i=1,2,3$).
From the formula (4.2), we deduce that
$$|\mathop{\mathrm{det}}T_{X_{4}^{'}}|=\frac{64}{n(\Phi_{4}^{'})^{2}}=4.$$
By the same reason of the case $d$ is 3,
we conclude $X_{4}^{'}$ is isomorphic to $X_{4}$.
\end{proof}
\section{Generalization of the construction}
In this section, we generalize a construction of 6.2 in terms of singularity of branching locus.
\\[2ex]
Let $C$ be a sextic curve and $p_{i} \; (i=1,2,3)$ and $q$ are points in $\mathbb{P}^{2}$.
We denote by $l_{i}$ the line passing through $p_{i}$ and $q$.
Now, we consider the conditions as follows:
\\[1ex]
(i) \indent
 $\{p_{i}, q |\; i=1, 2, 3 \}$ is in general position.
\\[1ex]
(ii)\indent
$C$ has a $D_{4}$ singularity at $q$ and an $A_{3}$ singularity at $p_{i}$, and is smooth elsewhere.
\\[1ex]
(iii)\indent The multiplicity of $C$ and $l_{i}$ at $p_{i}$ (resp. $q$) is 2 (resp. 4)
\begin{ex}
Let $D_{\mu}$ $(\mu \neq 0,-4)$ be a sextic curve given by the equation:
$$(x_{0}^2x_{1}+x_{1}^2x_{2}+x_{2}^2x_{0}-3x_{0}x_{1}x_{2})^2+\mu x_{0}x_{1}x_{2}(x_{0}-x_{1})(x_{1}-x_{2})(x_{2}-x_{0})=0.$$
Let $p_{i}$ $(i=1,2,3)$ be vertices of the triangle  $x_{0}x_{1}x_{2}=0$,
and let $q$ be $(1,1,1)$.
Then $D_{\mu}$, $p_{i}$ and $q$ satisfy the conditions as above.
\end{ex}
\begin{rem}
For the case of $\mu =-4$ (resp. $\mu =\infty$), $D_{-4}=C_{7}$ (resp. $D_{\infty}=C_{4}$),
where $C_{d} \,\,\, (d=4, 7)$ is a branching locus of $X_{d}$.
\end{rem}
\begin{lem}\label{13}
Let X be a K$3$ surface. The following conditions are equivalent.
\\
$(1)$
$X$ is isomorphic to $X_{C}$ and $r(\Phi_{C})=0$ and $n(\Phi_{C})=1,$
 where $C$ satisfies the conditions $(i)$, $(ii)$ and $(iii)$ as above.
\\
$(2)$
There exists an elliptic fibration $f:X\to \mathbb{P}^{1}$ such that $f$ has three singular fibres of types $I_{6}$, and others are $I_{1}$ or $I\hspace{-,1em}I$,
and $r(f)=0$ and $n(f)=1$.
\end{lem}
\begin{proof}
(1)$\Rightarrow$(2):
Since $C$ has a $D_{4}$ singularity at $q$,
there exists an elliptic fibration $f:X_{C} \to \mathbb{P}^{1}$.
We define $\widetilde{\mathbb{P}}^{2}$, $E$, $\widetilde{C}$ and $\pi: X_{C} \to \mathbb{P}^{2}$ in the same way as 6.1.
We denote by $l_{i}$ the line passing through $p_{i}$ and $q$,
and $\widetilde{l_{i}}$ the strict transform of $l_{i}$ under the blowing-up $\pi :X_{C} \to \mathbb{P}^{1}$.
Let $\widetilde{p_{i}}$ (resp. $\widetilde{q_{i}}$) be the inverse image of $p_{i}$ (resp. the intersection of $\widetilde{l_{i}}$
and $E$).
By the condition (ii) and (iii),
$\widetilde{C}$ has an $A_{1}$ singularity at $\widetilde{q_{i}}$ and an $A_{3}$ singularity at $\widetilde{p_{i}}$.
Hence $\pi^{-1}(l_{i})$ is a singular fibre of type $I_{6}$.
\begin{figure}[H]
 \centering
 \includegraphics[width=8cm]{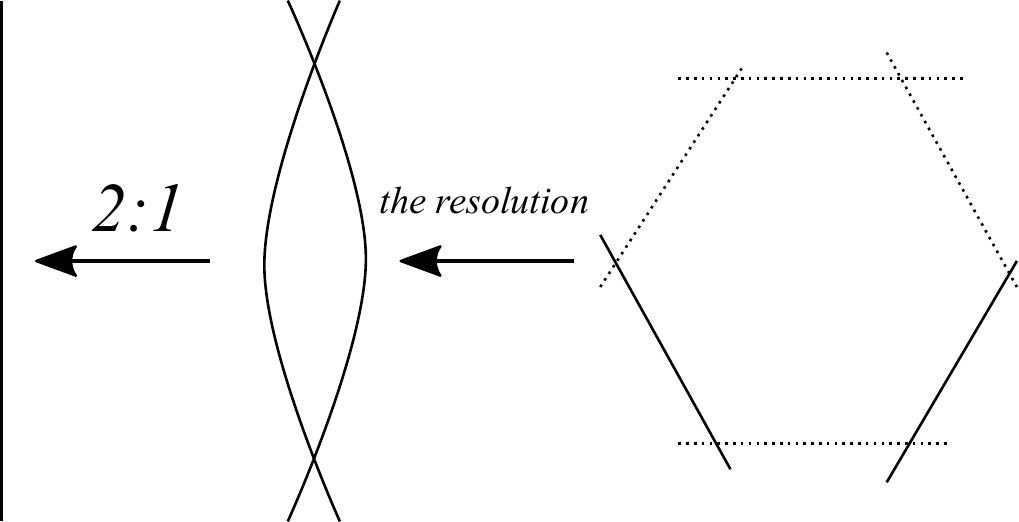}
 \caption{}
 \label{resolution}
\end{figure}
\noindent
Let $l$ be a line passing through $q$ with $l \neq l_{i}$,
and $\widetilde{l}$ the strict transform of $l$ under the blowing-up $\pi :X_{C} \to \mathbb{P}^{1}$.
Since $\widetilde{l}$ does not meet $\widetilde{C}$ at singular points of $\widetilde{C}$,
the number of component of $\pi^{-1}(l)$ is 1.
Hence,
if it is a singular fibre, it is of types $I_{1}$ or $I\hspace{-,1em}I$.
\\[1ex]
(2)$\Rightarrow$(1):
The same argument as Naruki \cite{Naruki} can be applied.
Let $\iota:X\to X$ be a inversion involution.
$X/\iota$ has the canonical projection $p:X/\iota \to \mathbb{P}^{1}$. 
On the surface $X/\iota$,
we have three configurations of the following figures coming from fibres of type $I_{6}$:
\begin{figure}[H]
 \centering
 \includegraphics{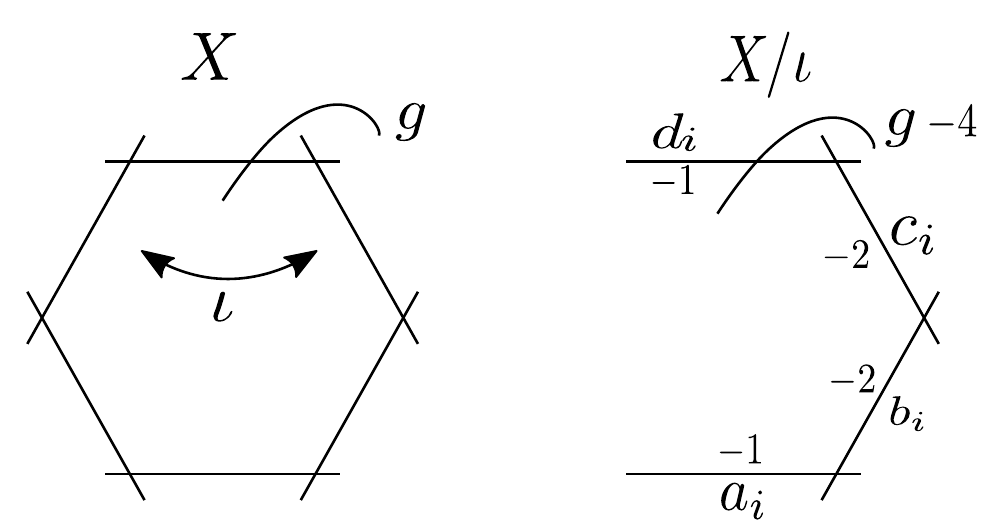}
 \caption{}
 \label{acton}
\end{figure}
\noindent
Here, $a_{i}$, $b_{i}$, $c_{i}$ and $d_{i}$ are the curves,
and $-1,-2,-4$ are the self-intersection numbers.
We blow down $a_{i}$ and $d_{i}$ ($i=1,2,3$),
and then the image of $b_{i}$ is exceptional curve.
Blowing down the image of $b_{i}$,
we obtain a $\mathbb{P}^{1}$-bundle over $\mathbb{P}^{1}$.
Since the self-intersection number of the image of $g$ is $-1$,
it is a tautological bundle.
Hence if we blow down the image of $g$,
we obtain the projective plane $\mathbb{P}^{2}$ and
the map from $X$ to $\mathbb{P}^{2}$, denoted by $\pi:X \to \mathbb{P}^{2}$.
Now we define $p_{i}$ and $q$ as follows:
$$p_{i}=\pi(a_{i})=\pi(b_{i}),\; q=\pi(g).$$

By the configuration of three singular fibres, the branching locus $C$ is a curve on $\mathbb{P}^{2}$ which has a $D_{4}$ singularity at $q$ and an $A_{3}$ singularity at $p_{i}$.
The image of fibres of $f$ are lines passing through $q$.
Let $l$ be a general line passing through $q$.
Since an elliptic curve $\pi^{-1}(l)$ intersects the fixed locus of $\iota$ exactly three times outside a global section $g$,
the line $l$ intersects the branching locus $C$ exactly three times outside triple point $q$.
By Bezout theorem, $C$ is a sextic curve.
\\
If $\{p_{i} | i=1, 2, 3\}$ is not in general position, $X$ has three global sections.
This contradicts  $n(f)=1$. Therefore, $\{p_{i} | i=1, 2, 3\}$ is in general position.
Hence $X$ is isomorphic to the minimal resolution of the double covering branched over $C$ which satisfies the condition (i), (ii) and (iii).
\end{proof}
\begin{cor}\label{14}
If X satisfies the conditions $(1)$ and $(2)$ of Lemma $\ref{13}$, we have $NS(X)=U\oplus A_{5}^{\oplus 3}$.
\end{cor}
\begin{proof}
This is an immediate consequence of (2) of Lemma \ref{13} and the equation (\ref{eq:MW}).
\end{proof}
\begin{lem}\label{15}
Let $L$ be a unimodular even lattice with the signature $(3,19)$.
There exists a primitive sublattice of $L$, denoted by $T$, with the signature $(2,3)$ such that $q_{T}=-q_{(U\oplus A_{5}^{\oplus 3})}$.
\end{lem}
\begin{proof}
Let $r_{i}$ ($i=1,2,...n$) be the basis of the lattice $A_{n}$ which correspond to vertices of the Dynkin diagram $A_{n}$.
It is obvious that
$$A_{n-1}\hookrightarrow A_{n}\; :\; r_{i}\mapsto r_{i}\; (i=1,2...,n-1).$$
is a primitive embedding.
This induces a primitive embedding:
$$ U\oplus A_{5}^{\oplus 3}\hookrightarrow U\oplus A_{6}^{\oplus 3}.$$
We consider the case of $X_{7}$.
We can find the desired lattice $T$ as a sublattice of $H^{2}(X_{7},\mathbb{Z})(\cong L)$ as follows.
The lattice $U\oplus A_{6}^{\oplus 3}$ is of finite index in $NS(X_{7})$.
By lemma 2.1, the lattice $U\oplus A_{5}^{\oplus 3}$ is primitive in $NS(X_{7})$.
Then, $U\oplus A_{5}^{\oplus 3}$ is also primitive in $H^{2}(X_{7},\mathbb{Z})$.
We denote by $T$ the orthogonal complement of $U\oplus A_{5}^{\oplus 3}$ in $H^{2}(X_{7},\mathbb{Z})$.
$T$ satisfies the conditions.
\end{proof}
\begin{lem}\label{16}
Let L be a unimodular even lattice with the signature $(3,19)$, and
let T be a lattice as in Lemma $\ref{15}$.
Then there exists uniquely up to $O(L)$ a primitive embedding $T \hookrightarrow L$.
\end{lem}
\begin{proof}
Let $T_{i}$ ($i=1,2$) be a primitive sublattice of $L$ which is isomorphic to $T$,
and let $S_{i}=T_{i}^{\bot}$.
By Lemma \ref{3}, $S_{i}$ is isomorphic to $U\oplus A_{5}^{\oplus 3}$.
Since $T_{1}\cong T_{2}$,
there exists an isomorphism $f: T_{1}\to T_{2}$.
$f$ induces the map $\overline{f}:A_{T_{1}}\to A_{T_{2}}$.
We define a map $\phi :A_{S_{1}}\to A_{S_{2}}$ as follows:
$$
 \begin{CD}
   A_{T_{1}} @>{\overline{f}}>>A_{T_{2}}\\
  @V{r_{T_{1},S_{1}}}VV @V{r_{T_{1},S_{1}}}VV \\
    A_{S_{1}} @>{\phi}>>A_{S_{2}}
 \end{CD}
$$
By Lemma \ref{3}, there exists $g:S_{1} \to S_{2}$ such that $\overline{g}=\phi$.
It is obvious $$(g,f):S_{1}\oplus T_{1}\to S_{2}\oplus T_{2}$$ extends to an isomorphism from $S_{1}^{*}\oplus T_{1}^{*}$ to $S_{2}^{*}\oplus T_{2}^{*}$.
Since $(g,f)$ preserves $L$, $(g,f)|L$ is an isometry of $L$, and the image of $T_{1}$ by $(g,f)$ is $T_{2}$.
Hence we conclude the consequence of Lemma \ref{16}.
\end{proof}
We define a period domain $\mathcal{B}$ as follows:
$$\mathcal{B}:=\{ \omega \in \mathbb{P}(T\otimes \mathbb{C}) | \hspace{+,5em} \langle\omega , \omega \rangle=0, \langle\omega ,\overline{\omega} \rangle>0 \} .$$
Let $X$ be a K3 surface of which a nowhere vanishing $2$-form corresponds to $\omega$.
If there exists an element $\delta \in T$ such that $\langle \delta, \omega \rangle =0$,
$\delta$ is an element of $S_{X}$.
Then $T_{X}$ is a proper subset of $T$.
This is not a general case.
To exclude this case, we define $\mathcal{H_{\delta}} \; \mbox{and} \; \mathcal{H}$ as follows:
$$\mathcal{H_{\delta}}:=\{ \omega \in \mathcal{B} | \hspace{+,5em} \langle \omega, \delta\rangle=0\} \; 
\mbox{for}\; \delta \in T, $$
$$\mathcal{H}:=\cup\mathcal{H_{\delta}}.$$
\begin{thm}\label{17}
Let $(X, \alpha_{X})$ be a marked $K3$ surface and $\omega_{X}$ be a nowhere vanishing holomorphic $2$-form of $X$.
Then the following conditions are equivalent.
\\
$(i)$ The condition $(1)$ of Lemma $\ref{13}$,
\\
$(ii)$ The condition $(2)$ of Lemma $\ref{13}$,
\\
$(iii)$
$\alpha_{X}(\omega_{X}) \in \mathcal{B} \hspace{-,1em}\setminus \hspace{-,2em}\mathcal{H}$.
\end{thm}
\begin{proof}
(i)$\Leftrightarrow$(ii): 
This is Lemma \ref{13}.
\\
(ii)$\Rightarrow$(iii):
This is an immediate consequence of Corollary \ref{14} , lemmas \ref{3}, \ref{15} and \ref{16}.
\\
(iii)$\Rightarrow$(ii):
By $\alpha_{X}(\omega_{X}) \in \mathcal{B} \hspace{-,1em} \setminus \hspace{-,2em} \mathcal{H}$,
we have $NS(X)=U\oplus A_{5}^{\oplus 3}$ and $T_{X}=T$.
Now, we denote by $E$ and $F$ the basis of the hyperbolic lattice $U$ and by $\Theta_{i}^{k}$ $(i=1,...,5  ;k=1,2,3)$ the basis of $k$-th component of $A_{5}^{\oplus 3}$,
and define $\Theta_{0}^{k}$ as $E-{\sum_{i=1}^{5} \Theta_{i}^{k}}$.
There exists $\sigma \in W(X)$ such that $ E^{'}=\sigma (E)\in \overline{D}(X)$ and the linear system $|E^{'}|$  contains  an elliptic curve.
By lemma 3.1, $\sigma (\Theta_{i}^{k})$ is represented by an effective divisor $D_{i}^{k}$.
By Bertini's theorem, we obtain an elliptic fibration  $\phi_{|E^{'}|}:X\to \mathbb{P}^{1}$.
It is obvious that $F_{k}:={\sum_{i=1}^{5} D_{i}^{k}}\in |E^{'}|$ is a singular fibre.
Since the number of the components of $F_{k}$ is larger than six, the type of $F_{k}$ is either $I_{n}$ or $I_{n}^{*}$  $(n\geq 6)$.
Let $m_{k}^{1}$ be the number of irreducible component of $F_{k}$.
We deduce from the formula (4.1) that $m_{1}+m_{2}+m_{3} \leq 18$.
Therefore,  $m_{k}=6$ and $r(\phi_{|E^{'}|})=0$.
By the formula (4.2), we deduced that the condition (ii).
\end{proof}
\begin{thm}\label{18}
Let $\omega_{X_{d}}$ $(d=3,4,7)$  be a nowhere vanishing holomorphic $2$-forms of $X_{d}$.
\\
Then $\omega_{X_{d}} \in \mathcal{H} \subset \mathcal{B}$.
\end{thm}
\begin{proof}
It is enough to show that there exists  a primitive embedding:
$$T_{X_{d}} \subset T\:(d=4, 7).$$
In the case where $d=7$, this is an immediate consequence of the proof of Lemma \ref{15}.
In the same way, we can find a sublattice of $D_{6}$ (resp. $E_{6}$) isomorphic to $A_{5}$:
\begin{figure}[H]
 \centering
 \includegraphics{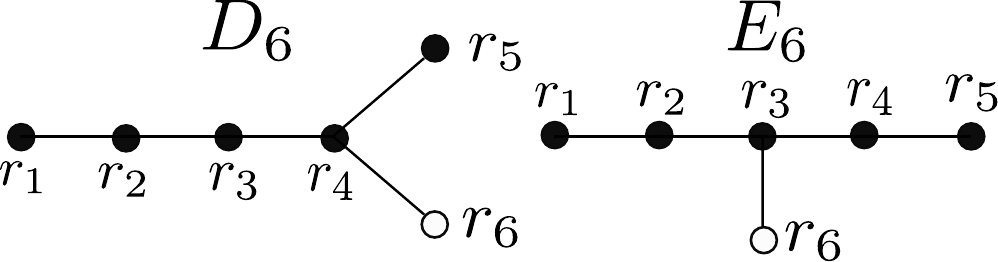}
 \caption{}
 \label{Dynkin}
\end{figure}%
\noindent
Here, the black circles consist the Dynkin diagram of type $A_{5}$.
Hence, in the same way of Lemma \ref{15}, we have primitive embeddings:
\begin{equation}
   \begin{split}
&U\oplus A_{5}^{\oplus 3} \hookrightarrow NS(X_{3}) \hookrightarrow L \\
&U\oplus A_{5}^{\oplus 3} \hookrightarrow NS(X_{4}) \hookrightarrow L
   \end{split}
\end{equation}
If we take the orthogonal complements of these lattices in $L$,
we have the desired primitive embeddings $T_{X_{d}} \subset T\: (d=3,4)$.
\end{proof}
\section{The automorphism group}
\begin{thm}
Let $X$ be a $K3$ surface which satisfies the conditions of theorem $\ref{17}$.
Then its automorphism group $Aut(X)$ is an infinite group.
\end{thm}
\begin{proof}
The same argument as Shioda and Inose \cite{Shioda} can be applied.
Let $f:X \to \mathbb{P}^{1}$ be an elliptic K3 surface.
If $r(f)$ is larger than $1$,
by the equation (\ref{eq:MW}),
there are infinitely many sections of $f$.
Then the automorphism group of $X$ is an infinite group.
\\
So it is enough to show that there exists an elliptic fibration $f:X\to \mathbb{P}^{1}$ with $r(f) \geq 1$.
By proposition \ref{17}, $X$ has an elliptic fibration $\Phi$ which satisfies the condition (ii) of Lemma \ref{13}.
The elliptic fibration $\Phi$ has a global section $g$ and the three singular fibre of type $I_{6}$.
The components of a fibre of type $I_{6}$ is denoted by $\Theta^{k}_{l}$ ($k=1,2,3$: $l=0,1,...,5$) which satisfies the following conditions:
$$\langle \Theta^{k}_{l}, \Theta^{k}_{l+1} \rangle=1,$$
$$\langle \Theta^{k}_{0}, g \rangle =1.$$
Now, we define an effective divisor E and $E_{k}$ as follows:
$$E:=3g+2(\Theta_{0}^{1}+\Theta_{0}^{2}+\Theta_{0}^{3})+(\Theta_{1}^{1}+\Theta_{1}^{2}+\Theta_{1}^{3}),$$
$$E_{k}:=E-(\Theta_{3}^{k}+\Theta_{4}^{k}).$$
By Lemma \ref{8}, $X$ has an elliptic fibration $f:X\to \mathbb{P}^{1}$ with  a singular fibre $E$. The type of $E$ is $I\hspace{-,1em}V^{*}$.
By Lemma \ref{8}, $\Theta^{k}_{2}$ are global sections of $f$.
We consider the three divisors:
$$\Theta^{k}_{3}+\Theta^{k}_{4} \; (k=1,2,3).$$
These divisors do not intersect the fibre $E$ of $f$.
Hence, the image $f(\Theta^{k}_{3}+\Theta^{k}_{4})$ is a point $t_{k}$ of $\mathbb{P}$,
and $(-2)$-curves $\Theta^{k}_{3}$ and $\Theta^{k}_{4}$ are the components of a singular fibre $f^{-1}(t_{k})$.
\\
Now, let $S$ be a lattice generated by the following divisors:
$$\Theta^{k}_{l} \; (k=1,2,3 \; ,l=0,1,3,4)\; and\; \Theta^{1}_{2}.$$
Let the divisors $\Theta^{k}$ be $\Theta^{k}_{1}+2\Theta^{k}_{2}+\Theta^{k}_{3}-\Theta^{k}_{5}$.
By an easy calculation, it is shown that a lattice $S^{\bot}$ which is the orthogonal complement of $S$ in $NS(X)$ is generated by these three divisors:
$$A=2E+\Theta^{1},$$
$$\Theta^{2}\; and\;  \Theta^{3}.$$
Moreover, the lattice $S^{\bot}$ is a lattice with the Gram matrix
$$\left(
\begin{array}{ccc}
-6 & 0 & 0 \\
0  &-6 & 0 \\
0  & 0 & -6
\end{array}
\right).$$
\\[1ex]
Hence the elliptic fibration $f$ has the singular fibres $E$, $f^{-1}(t_{k})$, and others are $I_{1}$ or $I\hspace{-,1em}I$ because 
the lattice $S^{\bot}$ does not have $(-2)$-roots.
The types of singular fibres $f^{-1}(t_{k})$ are $A_{2}$, $A_{3}$ or $D_{4}$ for the same reason.
If $r(f)=0$, the trivial sublattice is isomorphic to $U\oplus E_{6}\oplus A_{3}^{\oplus 3}$ or 
$U\oplus A_{2}\oplus A_{3} \oplus D_{4}$.
Here,  
\begin{equation}
   \begin{split}
&\frac{d(NS(X))}{d(U\oplus E_{6}\oplus A_{3}^{3})} \\
&=\frac{d(U\oplus A_{5}^{3})}{d(U\oplus E_{6}\oplus A_{3}^{3})} \\
&=\frac{6^{3}}{3\cdot 4^{3}}
   \end{split}
\end{equation}
In the same way,
$$ \frac{d(NS(X))}{d(U\oplus A_{2}\oplus A_{3} \oplus D_{4})}=\frac{6^{3}}{3 \cdot 4^{2}}$$
These are not squares,
and contradict the formula of Lemma \ref{2}.
\end{proof}

\end{document}